\documentclass[11pt]{article}
\usepackage{amsfonts}
\usepackage{amsmath}
\usepackage{amssymb}
\usepackage{amsthm}
\usepackage{latexsym}
\usepackage{hyperref}
\usepackage{amsbsy}
\usepackage{bbm}
\usepackage{fullpage}

\newtheorem{theorem}{Theorem}

\newtheorem{lemma}[theorem]{Lemma}

\begin{document}

\title{Characterizing graphs of maximum principal ratio}
\author{Michael Tait and Josh Tobin\thanks{Both authors were partially supported by NSF grant DMS 1362650 \newline \hspace*{4mm} \{\texttt{mtait, rjtobin}\}\texttt{@math.ucsd.edu}}}

\maketitle

\abstract{The principal ratio of a connected graph, denoted $\gamma(G)$,
is the ratio of the maximum and minimum entries of its first eigenvector.
Cioab\u{a} and Gregory conjectured that the graph on $n$ vertices
maximizing $\gamma(G)$ is a kite graph: a complete graph with
a pendant path.  In this paper we prove their conjecture.  
}
 
\section{Introduction}

Several measures of graph irregularity have been proposed to evaluate how far a graph is
from being regular.  In this paper we determine the extremal graphs
with respect to one such irregularity measure,
answering a conjecture of Cioab\u{a} and Gregory \cite{CioabaGregory2007}.

All graphs in this paper will be simple and undirected, and all
eigenvalues are of the adjacency matrix of the graph.
For a connected graph $G$, the eigenvector corresponding to its largest
eigenvalue, the \textit{principal eigenvector}, can be taken to have
all positive entries. If $\textbf{x}$ is this eigenvector, let
$x_{\text{min}}$ and $x_{\text{max}}$ be the smallest and largest eigenvector
entries respectively.  Then define the \textit{principal ratio},
$\gamma(G)$ to be
 \[ \gamma(G) = \frac{x_{\text{max}}}{x_{\text{min}}} .\]
Note that $\gamma(G) \geq 1$ with equality exactly when $G$ is regular,
and it therefore can be considered as a measure of graph irregularity.

Let $P_r \cdot K_s$ be the graph attained by identifying an end vertex
of a path on $r$ vertices to any vertex of a complete graph
on $s$ vertices.  This has been called a \textit{kite graph} or a
\textit{lollipop graph}.  Cioab\u{a} and Gregory \cite{CioabaGregory2007}
conjectured that the connected graph on $n$ vertices maximizing $\gamma$
is a kite graph.  Our main theorem proves this conjecture for $n$ large
enough.

\begin{theorem}\label{main_theorem}
  For sufficiently large $n$, the connected graph $G$ on $n$
  vertices with largest principal ratio is a kite graph.
\end{theorem}

We note that Brightwell and Winkler \cite{BrightwellWinkler1990} showed that
a kite graph maximizes the expected hitting time of a random walk.
Other irregularity measures for graphs have been well--studied.
Bell \cite{Bell1992} studied the irregularity measure $\epsilon(G) := \lambda_1(G) - \bar{d}(G)$,
the difference between the spectral radius and the average degree of $G$.
He determined the extremal graph over all (not necessarily connected)
graphs on $n$ vertices and $e$ edges.  It is not known what the extremal
connected graph is, and Aouchiche et al \cite{Aouchiche2008} conjectured that this extremal
graph is a `pineapple': a complete graph with pendant vertices added
to a single vertex.  Bell also studied the
\textit{variance} of a graph,
 \[var(G) = \frac{1}{n} \sum_{v\in V(G)} \left| d_v - \bar{d} \right|^2 . \]
 Albertson \cite{Albertson1997} defined a measure of irregularity by 
 \[
 \sum_{uv\in E(G)} |d(u) - d(v)|
 \]
 and the extremal graphs were characterized by Hansen and M\'elot \cite{HansenMelot2002}. 
 
Nikiforov \cite{Nikiforov2006} proved several inequalities comparing
$var(G)$, $\epsilon(G)$ and $s(G) := \sum_v |d(u) - \bar{d}|$.  
Bell showed that $\epsilon(G)$ and $var(G)$ are incomparable in general
\cite{Bell1992}.  Finally, bounds on $\gamma(G)$ have been given in
\cite{CioabaGregory2007, PapendieckRecht2000, Minc1970, Latham1995, Zhang2005}.

\section{Preliminaries}

Throughout this paper $G$ will be a connected simple graph on $n$ vertices.
The eigenvectors and eigenvalues of $G$ are those of the adjacency
matrix $A$ of $G$.  The vector $v$ will be the eigenvector corresponding
to the largest eigenvalue $\lambda_1$,  and we take $v$ to be scaled
so that its largest entry is $1$.  Let $x_1$ and $x_k$
be the vertices with smallest and largest eigenvector entries respectively, and
if several such vertices exist then we pick any of them arbitrarily.
Let $x_1, x_2, \cdots, x_k$ be a shortest path between $x_1$ and
$x_k$.  Let $\gamma(G)$ be the principal ratio of $G$.  We will abuse
notation so that for any vertex $x$, the symbol $x$ will refer also to $v(x)$,
the value of the eigenvector entry of $x$.  For example, with this notation
the eigenvector equation becomes

 \[ \lambda v = \sum_{w \sim v} w. \]

\noindent We will make use of the Rayleigh quotient characterization of the
largest eigenvalue of a graph,

\begin{equation}\label{rayleigh}
  \lambda_1(G) = \max_{0 \neq v} \frac{v^T A(G) v}{v^t v} 
\end{equation}

Recall that the vertices $v_1, v_2, \cdots, v_m$ are a 
\textit{pendant path} if the induced graph on these vertices
is a path and furthermore if, in $G$, $v_1$ has degree $1$ and
the vertices $v_2, \cdots, v_{m-1}$ have degree $2$
(note there is no requirement on the degree of $v_m$).

\begin{lemma}\label{path_bound}
  If $\lambda_1 \geq 2$ and $\sigma = (\lambda_1 + \sqrt{\lambda_1^2 - 4})/2$, then for
  $1 \leq j \leq k$,
   \[ \gamma(G) \leq \frac{\sigma^j - \sigma^{-j}}{\sigma - \sigma^{-1}} x_j^{-1}. \]
Moreover we have equality if the vertices $x_1, x_2, \cdots, x_{j}$ are a pendant path.
  
\end{lemma}
\begin{proof}
  We have the following system of inequalities
   \begin{eqnarray*}
     \lambda_1 x_1 & \geq & x_2 \\
     \lambda_1 x_2 & \geq & x_1 + x_3 \\
     \lambda_1 x_3 & \geq & x_2 + x_4 \\
     \vdots & & \vdots\\
     \lambda_1 x_{j-1} & \geq & x_j + x_{j-2}
   \end{eqnarray*}
 The first inequality implies that
  \[ x_1 \geq \frac{1}{\lambda_1} x_2\]
 Plugging this into the second equation and rearranging gives
  \[ x_2 \geq \frac{\lambda_1}{\lambda_1^2 - 1} x_3 \]
 Now assume that
  \[ x_i \geq \frac{u_{i-1}}{u_i} x_{i+1}. \]
 with $u_j$ positive for all $j<i$.  Then
  \[ \lambda_1 x_{i+1} \geq x_{i} + x_{i+2} \]
 implies that
  \[ x_{i+1} \geq \frac{u_i}{\lambda_1 u_i - u_{i-1}} x_{i+2}. \]
 where $\lambda_1 u_i - u_{i-1}$ must be positive because  $x_j$
 is positive for all $j$.
 Therefore the coefficients $u_i$ satisfy the recurrence
  \[ u_{i+1} = \lambda_1 u_i - u_{i-1}\]
 Solving this and using the initial conditions $u_0 = 1$,
 $u_1 = \lambda$ we get
  \[ u_i = \frac{\sigma^{i+1} - \sigma^{-i-1}}{\sigma - \sigma^{-1}} \]
 In particular, $u_i$ is always positive, a fact implicitly
 used above.  Finally this gives,
  \[ x_1 \geq \frac{u_0}{u_1} x_2 \geq \frac{u_0}{u_1} \cdot \frac{u_1}{u_2} x_3 \geq \cdots \geq \frac{x_j}{u_{j-1}} \]
 Hence
  \[ \gamma(G) = \frac{x_k}{x_1} = \frac{1}{x_1} \leq \frac{\sigma^j - \sigma^{-j}}{\sigma - \sigma^{-1}} x_j^{-1} \]
 If these vertices are a pendant path, then we have equality throughout.
\end{proof}

We will also use the following lemma which comes from the
paper of Cioab\u{a} and Gregory \cite{CioabaGregory2007}.

\begin{lemma}\label{kite_lambda}
  For $r \geq 2$ and $s \geq 3$,
   \[ s - 1 + \frac{1}{s(s-1)} < \lambda_1(P_r \cdot K_s) < s - 1 + \frac{1}{(s-1)^2} . \]
\end{lemma}

In the remainder of the paper we prove Theorem~\ref{main_theorem}.
We now give a sketch of the proof that is contained in Section~\ref{sec_proof}.

\begin{enumerate}
 \item We show that the vertices $x_1, x_2, \cdots, x_{k-2}$ are
   a pendant path and that $x_k$ is connected to all of the vertices
   in $G$ that are not on this path (lemma~\ref{connect_every}).
 \item Next we prove that the length of the path is approximately
   $n - n/\log(n)$ (lemma~\ref{s_range}).
 \item We show that $x_{k-2}$ has degree exactly $2$ (lemma~\ref{k_2_lemma}), which
   extends our pendant path to $x_1, x_2, \cdots, x_{k-1}$.
   To do this, we find conditions under which adding or deleting
   edges increases the principal ratio (lemma~\ref{change}).
 \item Next we show that $x_{k-1}$ also has degree exactly $2$ (lemma~\ref{k_1_lemma}).
   At this point we can deduce that our extremal graph is either
   a kite graph or a graph obtained from a kite graph
   by removing some edges from the clique.  We show that
   adding in any missing edges will increase the principal ratio,
   and hence the extremal graph is exactly a kite graph.
   
\end{enumerate}

\section{Proof of Theorem~\ref{main_theorem}}\label{sec_proof}

Let $G$ be the graph with maximal principal ratio among all connected
graphs on $n$ vertices, and let $k$ be the number of vertices in a
shortest path between the vertices with smallest and largest eigenvalue
entries. As above, let $x_1,\cdots, x_k$ be the vertices of the shortest path, where $\gamma(G) = x_k / x_1$.  Let $C$ be the set of vertices not on this shortest
path, so $|C| = n-k$.  Note that there is no graph with $n-k=1$, as the endpoints of a path have the same principal eigenvector entry.  Also
$\lambda_1(G) \geq 2$, otherwise $P_{n-2} \cdot K_3$ would have larger
principal ratio.  Finally note that $k$ is strictly larger than $1$,
otherwise $x_k = x_1$ and $G$ would be regular.

\begin{lemma}\label{max_lambda}
  $\lambda_1(G) > n-k$.
\end{lemma}
\begin{proof}
  Let $H$ be the graph $P_k \cdot K_{n-k+1}$. It is straightforward to see that in $H$, the smallest entry of the principal eigenvector is the vertex of degree $1$ and the largest is the vertex of degree $n-k+1$. Also note that in $H$, the vertices on the path $P_k$ form a pendant path.
  By maximality we know that $\gamma(G) \geq \gamma(H)$.
  Combining this with lemma~\ref{path_bound}, we get
   \[ \frac{\sigma^k - \sigma^{-k}}{\sigma - \sigma^{-1}} \geq \gamma(G) \geq \gamma(H) =  \frac{\sigma_H^{k} - \sigma_H^{-k}}{\sigma_H - \sigma_H^{-1}} \]
  where $\sigma_H = \left(\lambda_1(H) + \sqrt{\lambda_1(H)^2 - 4}\right) /2$.

  \noindent Now the function
   \[ f(x) = \frac{x^k - x^{-k}}{x - x^{-1}}\]
  is increasing when $x \geq 1$.
  Hence we have
  $\sigma \geq \sigma_H$, and so
  $\lambda_1(G) \geq \lambda_1(H) > n-k$.
\end{proof}

\begin{lemma}\label{connect_every}
  $x_1, x_2, \cdots, x_{k-2}$ are a pendant path in $G$, and $x_k$
  is connected to every vertex in $G$ that is not on this
  path.
\end{lemma}
\begin{proof}
  By our choice of scaling, $x_k = 1$.  From lemma~\ref{max_lambda}
   \[ n-k < \lambda_1(G) = \sum_{y \sim  x_k} y \leq |N(x_k)|. \]
  Now $|N(x_k)|$ is an integer, so we have $|N(x_k)| \geq n-k+1$.
  Moreover because $x_1, x_2, \cdots, x_k$ is an induced path, we
  must have that $|N(x_k)| = n-k+1$ exactly, and hence the
  $N(x_k) = C \cup \{ x_{k-1} \}$.  It follows that $x_1, x_2, \cdots, x_{k-3}$
  have no neighbors off the path, as otherwise there would be a shorter
  path between $x_1$ and $x_k$.
\end{proof}

\begin{lemma}\label{s_range}
For the extremal graph $G$, we have $n-k = (1+o(1))\frac{n}{\log n}$.
\end{lemma}
\begin{proof}
Let $H$ be the graph $P_j \cdot K_{n-j+1}$ where $ j = \left\lfloor n - \frac{n}{\log n}\right\rfloor$, and let $G$ be the connected graph on $n$ vertices with maximum principal ratio. Let $x_1,\cdots, x_k$ be a shortest path from $x_1$ to $x_k$ where $\gamma(G) = \frac{x_k}{x_1}$. By lemma \ref{connect_every}, we have
\[
\lambda_1(G) \leq \Delta(G) \leq n-k+1.
\]
By the eigenvector equation, this gives that
\begin{equation}\label{gamma of G}
\gamma(G) \leq (n-k+1)^k
\end{equation}
Now, lemma \ref{path_bound} gives that
\[
\gamma(H)  = \frac{\sigma_H^j - \sigma_H^{-j}}{\sigma_H - \sigma_H^{-1}},
\]
where
\[
\sigma(H) = \frac{\lambda_1(H) + \sqrt{\lambda_1(H)^2 -4}}{2}.
\]

Now, $s-1 + \frac{1}{s(s-1)} < \lambda_1(P_r\cdot K_s) < s-1 + \frac{1}{(s-1)^2}$, so we may choose $n$ large enough that $\frac{n}{\log n} + 1 >\sigma_H - \sigma_H^{-1} > \frac{n}{\log n}$. By maximality of $\gamma(G)$, we have
\[
(n-k+1)^k \geq \gamma(G) \geq \gamma(H) \geq \left(\frac{n}{\log n}\right)^{n-\frac{n}{\log n} - 2}.
\]
Thus, $n- k = (1+o(1))\frac{n}{\log n}$.
\end{proof}

For the remainder of this paper we will explore the structure of $G$ by
showing that if certain edges are missing, adding them would increase
the principal ratio, and so by maximality these edges must already be
present in $G$.  We have established that the vertices $x_1, x_2, \cdots, x_{k-2}$
are a pendant path, and so we have
\begin{equation}\label{pr_form}
 \gamma(G) = \frac{\sigma^{k-2} - \sigma^{-k+2}}{\sigma-\sigma^{-1}} \frac{1}{x_{k-2}}
\end{equation}
We will not add any edges that affect this path, and so the above equality will
remain true.   
The change in $\gamma$ is then completely determined by the change
in $\lambda_1$ and the change in $x_{k-2}$.  The next lemma gives conditions
on these two parameters under which $\gamma$ will increase or decrease.

\begin{lemma}\label{change}
Let $x_1, x_2, \cdots, x_{m-1}$ form a pendant path
  in $G$, where $n-m = (1+o(1)) n/\log(n)$.
  Let $G_+$ be a graph obtained from $G$ by adding some edges from $x_{m-1}$ to $V(G)\setminus \{x_1,\cdots, x_{m-1}\}$,
  where the addition of these edges does not affect which vertex
  has largest principal eigenvector entry.  Let
  $\lambda_1^+$ be the largest eigenvalue of $G_+$ with leading eigenvector
  entry for vertex $x$ denoted $x^+$, also normalized to have maximum
  entry one.  
  Define $\delta_1$ and $\delta_2$ such that
  $\lambda_1^+ = (1 + \delta_1) \lambda_1$ and
  $x_{m-1}^+ = (1 + \delta_2) x_{m-1}$.  Then 
  \begin{itemize}
     \item $\gamma(G_+) > \gamma(G)$ whenever $\delta_1 > 4 \delta_2 / n$
     \item $\gamma(G_+) < \gamma(G)$ whenever $\delta_1 \exp(2 \delta_1 \lambda_1 \log n ) <  \delta_2 / 3 n$.
  \end{itemize}
\end{lemma}
\begin{proof}
  We have
  \[\sigma = \lambda_1 - \lambda_1^{-1} - \lambda_1^{-3} - 2 \lambda_1^{-5} - \cdots - \frac{2}{2n-3} \binom{2n-2}{n} \lambda_1^{-(2n-1)}- \cdots\]
  So
   \[ \lambda_1^+ - \lambda_1 < \sigma_+ - \sigma < \lambda_1^+ - \lambda_1 - 2((\lambda_1^+)^{-1} - \lambda_1^{-1})\]
  when $\lambda_1$ is sufficiently large, which is guaranteed by lemma~\ref{s_range}.
  Plugging in $\lambda_1^+ = (1 + \delta_1) \lambda_1$, we get
   \[ \delta_1 \lambda_1 < \sigma_+ - \sigma < \delta_1 \lambda_1 + 2 \lambda_1^{-1}(1 - (1+\delta_1)^{-1}) < \delta_1 \lambda_1 + \delta_1\]
  In particular
   \[ (1+\delta_1/2) \sigma < \sigma_+ < (1+2 \delta_1) \sigma\]
  To prove part (i), we wish to find a lower bound in the change in the first factor of
  equation~\ref{pr_form}.  Let
   \[ f(x) = \frac{x^{m-1} - x^{-m+1}}{x-x^{-1}}. \]
  Then $2m x^{m-3} > f'(x) > (m-2) x^{m-3} - m x^{m-5}$, and using
  that $n-m \sim  n/\log(n)$ and $\sigma \sim \lambda_1$ which goes to infinity with $n$,
  we get $f'(x) \gtrsim (m-2) x^{m-3}$.  By linearization and because $f(\sigma) \sim \sigma^{m-2}$, it follows that
   \[ \frac{\sigma_+^{m-1} - \sigma_+^{-m+1}}{\sigma_+ - \sigma_+^{-1}} \geq \left(1 + \frac{\delta_1 (m-3)}{2}\right) \frac{\sigma^{m-1} - \sigma^{-m+1}}{\sigma - \sigma^{-1}}\]
  Hence, if
   \[ \frac{\delta_1 (m-3)}{2} > \delta_2\]
  then $\gamma(G_+) > \gamma(G)$.  In particular it
  is sufficient that $\delta_1 > 4\delta_2 / n$.

  To prove part (ii), recall from above that $f'(x) < 2m x^{m-3}$.
  Then, when $x = (1+o(1)) (n / \log(n))$
   \begin{eqnarray*}
     f'(x+\varepsilon) & < & 2m (x+\varepsilon)^{m-3} \\
     & = & 2m x^{m-3} \left( 1 + \frac{\varepsilon}{x} \right)^{m-3}\\
     & \leq & 2m x^{m-3} \exp\left(\frac{m \varepsilon}{x}\right) \\
     & \leq & 2n x^{m-3} \exp(2 \log(n) \varepsilon) 
   \end{eqnarray*}
  So for $0 < \varepsilon < \delta_1 \lambda_1$, we have
   \[ f'(x+\varepsilon) < 2n x^{m-3} \exp(2 \log(n) \delta_1 \lambda_1) \]
  Hence
   \[ \big( 1 + 3n \exp(2 \delta_1 \lambda_1 \log n ) \delta_1 \big) \frac{\sigma^{m-1} - \sigma^{-m+1}}{\sigma - \sigma^{-1}} >  \frac{\sigma_+^{m-1} - \sigma_+^{-m+1}}{\sigma_+ - \sigma_+^{-1}} \]

\end{proof}

\begin{lemma}\label{large_nbds}
  For every subset of $U$ of $N(x_k)$, we have
   \[ |U| - 1 < \sum_{y \in U} y \leq |U|. \]
  An immediate consequence is that there is at most one
  vertex in the neighborhood of $x_k$ with eigenvector
  entry smaller than $1/2$.
\end{lemma}
\begin{proof}
  The upper bound follows from $y \leq 1$, and the lower bound from
  the inequalities
   \[ \sum_{y \in N(x_k) \setminus U} y  \leq |N(x_k)| - |U|\]
  and
   \[ \sum_{y \in N(x_k)} y = \lambda_1(G) > |N(x_k)| - 1 .\]
\end{proof}

\begin{lemma}\label{k_2_lemma}
 The vertex $x_{k-2}$ has degree exactly $2$ in $G$.
\end{lemma}
\begin{proof}
  Assume to the contrary.  Let $U = N(x_{k-2}) \cap N(x_{k})$.  Then
  $|U| \geq 2$, so by lemma~\ref{large_nbds} we have
   \[ \sum_{y \in U} y > |U| - 1 \geq 1 . \]
  Now, by the same argument as the in the proof of lemma~\ref{path_bound},
  we have that
   \[ \gamma(G) = \frac{\sigma^{k-1} - \sigma^{-k+1}}{\sigma - \sigma^{-1}} \left( \sum_{y \in U} y \right)^{-1} \]
  Let $H = P_{k-1} \cdot K_{n-k+2}$.  Then by maximality of $\gamma(G)$ we have
  \begin{equation*}
   \frac{\sigma^{k-1} - \sigma^{-k+1}}{\sigma - \sigma^{-1}} > \gamma(G) \geq \gamma(H) = \frac{\sigma_H^{k-1} - \sigma_H^{-k+1}}{\sigma_H - \sigma_H^{-1}} 
  \end{equation*}
  So $\sigma > \sigma_H$, which means $\lambda_1(G) > \lambda_1(H) > n-k+1$.
  This means that $\Delta(G) > n-k+1$, but we have established that
  $\Delta(G) = n-k+1$.
\end{proof}

We now know that $x_1, x_2, \cdots, x_{k-1}$ is a pendant path in $G$, and
so equation~\ref{pr_form} becomes
\begin{equation}\label{pr_form2}
 \gamma(G) = \frac{\sigma^{k-1} - \sigma^{-k+1}}{\sigma-\sigma^{-1}} \frac{1}{x_{k-1}}
\end{equation}

\begin{lemma}\label{k_1_pre_lemma}
 The vertex $x_{k-1}$ has degree less than $11 |C| / \sqrt{\log n}$.  
\end{lemma}
\begin{proof}
  Assume to the contrary, so throughout this proof we assume that
  the degree of $x_{k-1}$ is at least $11 |C| / \sqrt{\log n}$.
  Let $G_+$ the graph obtained form $G$ with
  an additional edge from $x_{k-1}$ to a vertex $z \in C$ with $z \geq 1/2$.
  Let $\lambda^+_1 = \lambda_1(G_+)$ and  let $x^+$ be the principal eigenvector
  entry of vertex $x$ in $G_+$, where this eigenvector is normalized to have
  $x_{k}^+ = 1$.

  \noindent \textbf{Change in $\lambda_1$}: By equation~\ref{rayleigh}, we
  have $\lambda_1^+ - \lambda_1 \geq 2 \frac{x_{k-1} z}{||v||^2_2}$.
  A crude upper bound on $||v||_2^2$ is
   \[ ||v||_2^2 \leq 1 + \sum_{y \sim x_{k}} y + \frac{2}{\lambda_1} + \frac{4}{\lambda_1^2} + \cdots < 2 \lambda_1 \]
  We also have that $z \geq 1/2$ so
   \[ \lambda_1^+ \geq \left( 1 + \frac{x_{k-1}}{2\lambda_1^2} \right) \lambda_1 .\]

  \noindent \textbf{Change in $x_{k-1}$}:  Let $U = N(x_{k-1} \cap C)$.
  By the eigenvector equation we have
  \begin{eqnarray*}
    x_{k-1} & = & \frac{1}{\lambda_1} \left( x_{k-2} + x_{k} + \sum_{y \in U} y \right) \\
    x_{k-1}^+ & = & \frac{1}{\lambda_1^+} \left( x_{k-2}^+ + x_{k}^+ + z^+ + \sum_{y \in U} y^+ \right)    
  \end{eqnarray*}
  Subtracting these, and using that $\lambda_1 < \lambda_1^+$ and $x_k = x_k^+ = 1$, we get
   \[ x_{k-1}^+ - x_{k-1} \leq \frac{1}{\lambda_1} \left( x_{k-2}^+ - x_{k-2} + z^+ + \sum_{y \in U} y^+ - y\right) .\]
  By lemma~\ref{large_nbds}, we have $\sum_{y \in U} y^+ - y \leq 1$.  We also have
  $x_{k-2}^+ - x_{k-2} < 1$ and $z^+ \leq 1$.  Hence $x_{k-1}^+ - x_{k-1} \leq 3 / \lambda_1$,
  or
   \[ x_{k-1}^+ \geq \left( 1 + \frac{3}{\lambda_1 x_{k-1}} \right) x_{k-1}\]

  We can only apply lemma~\ref{change} if $x_{k}^+$ is the
  largest eigenvector entry in $G_+$.  So we must consider two cases.

  \noindent \textbf{Case 1:} If in
  $G^+$ the largest eigenvector entry is still attained by vertex $x_k$, then
  we can apply lemma~\ref{change}, and see that $\gamma(G^+) > \gamma(G)$
  if
   \[ \frac{x_{k-1}}{2 \lambda_1^2} \geq \frac{12}{\lambda_1 x_{k-1} n} \]
  or equivalently
   \[ x_{k-1}^2 \geq \frac{24 \lambda_1}{n} .\]
  We have that $\lambda_1 = (1+o(1)) (n - n / \log(n))$, so it suffices for
  \begin{equation}\label{eqn_need}
   x_{k-1} \geq \frac{5}{\sqrt{\log n}} .
  \end{equation}
  We know that
   \[ x_{k-1} > \frac{|U| - 1}{2 \lambda_1} .\]
  By assumption
   \[ |U| + 2 = N(x_{k-1}) \geq 11 |C| / \sqrt{\log n} \]
  Equation~\ref{eqn_need} follows from this, so $\gamma(G^+) > \gamma(G)$.  

  \noindent \textbf{Case 2:} Say the largest eigenvector entry of $G^+$ is no
  longer attained by vertex $x_k$.  It is easy to see that the largest
  eigenvector entry is not attained by a vertex with degree less than or equal
  to $2$, and comparing the neighborhood of any vertex in $C$ with the neighborhood
  of $x_{k}$ we can see that $x_k \geq y$ for all $y \in C.$  So the
  largest eigenvector entry must be attained by $x_{k-1}$.
  Then equation~\ref{pr_form2} no longer holds, instead we have
   \begin{equation}\label{case2} \gamma(G_+) = \frac{\sigma_+^{k-1} - \sigma_+^{-k+1}}{\sigma_+ - \sigma_+^{-1}} .\end{equation}
Recall that in lemma~\ref{change} we determined the change from $\gamma(G_+)$ to $\gamma(G)$ by considering $\lambda^+_1 - \lambda_1$ and $x^+_{k-1}-x_{k-1}$. In this case, by \eqref{case2}, we must consider $\lambda^+_1 - \lambda_1$ and $1-x_{k-1}$.
   Now if  $x_{k-1}^+ > x_k^+$ , then vertex $x_{k-1}$ in $G$
  is connected
  to all of $C$ except perhaps a single vertex.  Hence
  in $G$, the vertex $x_{k-1}$ is connected to all of $C$ except at most
  two vertices.  This gives the bound
   \[ 1 - x_{k-1} \leq 3 / \lambda_1 \]
  and so as in the previous case, $\gamma(G_+) > \gamma(G)$.

  So in all cases, $x_{k-1}$ is connected to all vertices in $C$ that have
  eigenvector entry larger than $1/2$.  If all vertices in $C$ have eigenvector
  entry larger than $1/2$, then $x_{k-1}$ is connected to all of $C$,
  and this implies that $x_{k-1} > x_k$, which is a contradiction.
  At most one vertex in $C$ is smaller than $1/2$, and so
  there is a single vertex $z \in C$ with $z < 1/2$.  We will
  quickly check that adding the edge $\{ x_{k-1}, z\}$ increases
  the principal ratio.  As before let $G_+$ be the graph obtained by
  adding this edge.  The largest eigenvector entry in $G_+$ is
  attained by $x_{k-1}$, as its neighborhood strictly contains the neighborhood
  of $x_{k}$.  
  As above, adding the edge $\{ z, x_{k}\}$ increases the spectral radius
  at least
   \[ \lambda_1^+ > \left(1 + \frac{z}{2\lambda_1^2} \right) \lambda_1 \]
  and we have $1 - x_{k-1} < 1 - z/\lambda_1$.  Applying lemma~\ref{change}
  we see that $\gamma(G_+) > \gamma(G)$, which is a contradiction.
  Finally we conclude that the degree of $x_{k-1}$ must be smaller than
  $11 |C| / \sqrt{\log n}$.
\end{proof}

We note that this lemma gives that $x_{k-1} < 1/2$ which implies that any vertex in $C$ has eigenvector entry larger than $1/2$.

\begin{lemma}\label{k_1_lemma}
 The vertex $x_{k-1}$ has degree exactly $2$ in $G$.  It follows that
 $x_{k-1} < 2 / \lambda_1$.
\end{lemma}
\begin{proof}
  Let $U = N(x_{k-1}) \cap C$, $c = |U|$.  If $c=0$ then we are done.
  Otherwise let $G_-$ be the graph obtained from $G$ by deleting these
  $C$ edges.  We will show that $\gamma(G_-) > \gamma(G)$.

  \noindent \textbf{(1) Change in $\lambda_1$:}
  We have by equation~\ref{rayleigh}, 
   \[ \lambda_1 - \lambda^{-}_1 \leq 2c \frac{x_{k-1}}{||v||_2^2}\]
  By Cauchy--Schwarz,
   \[ ||v||_2^2 > \sum_{x \in N(x_{k})} x^2 \geq \frac{\left(\sum_{x \in N(x_k)} x\right)^2}{|C|+1} \geq \frac{(n-k)^2}{n-k+1}\]
  We also have
   \[ x_{k-1} \leq \frac{c+2}{\lambda_1}\]
  Combining these we get
   \[ \lambda_1 - \lambda_1^{-} < \frac{9c^2}{\lambda_1 (n-k+1)} \Rightarrow \lambda_1 < \left(1 + \frac{9c^2}{\lambda_1 \lambda_1^{-} (n-k+1)}\right) \lambda_1^{-}\]
  We have $\lambda_1 \lambda_1^{-} > (n-k)^2$, so
  \[ \lambda_1 < \left( 1 + \frac{10c^2}{(n-k)^3} \right) \lambda_1^{-} \]
  
  \noindent \textbf{(2) Change in $x_{k-1}$:}
  At this point, we know that in $G_-$ the vertices $x_1,\cdots , x_{k}$ form a pendant path, and so by the proof of lemma~\ref{path_bound}, we have $x_{k-1}^- = (1+o(1)) / \lambda_1$. By the eigenvector equation and using that the vertices in $C$ have eigenvector entry at least $1/2$, we have $x_{k-1} > (1 + c/2) / \lambda_1$.  So
   \begin{equation*}
     x_{k-1} - x_{k-1}^{-} > \frac{1}{\lambda_1} \left( \frac{c}{2} + o(1) \right)
   \end{equation*}
   In particular,
    \[ x_{k-1} > \left( 1 + \frac{c}{3x_{k-1}^{-}\lambda_1}\right) x_{k-1}^-\]
   Applying lemma~\ref{change}, it suffices now to show that
   \begin{equation}\label{eqn_need2}
    \frac{10c^2}{(n-k)^3} \exp \left(2 \frac{10c^2}{(n-k)^3} \lambda_1^- \log n \right) < \frac{c}{9 x_{k-1}^- \lambda_1 n} .
   \end{equation}
   Now
    \[ \frac{10c^2}{(n-k)^3} < 10 \frac{11^2}{\log(n)} \frac{|C|^2}{(n-k)^3} < \frac{11^3}{\log n} \frac{\log n}{n} = \frac{11^3}{n} .\]
    Similarly $2 \frac{10c^2}{(n-k)^3} \lambda_1^- \log n < 2\cdot 11^3$, so the lefthand side of equation~\ref{eqn_need2} is smaller than $C_0 / n$, where
   $C_0$ is an absolute constant.
   For the righthand side, recall that $x_{k-1}^- \lambda_1 = 1 + o(1)$, and also that
    \[ c > \frac{11}{\sqrt{\log n}} \left( \frac{n}{\log n} + o(1) \right) > \frac{10n}{\log^{3/2} n} .\]
   So the righthand side is larger than $1 / \log^{3/2}{n}$.  Hence for large
   enough $n$, the righthand side is larger than the lefthand side.
  
\end{proof}

We are now ready to prove the main theorem.

{
\renewcommand{\thetheorem}{1}
\begin{theorem}
  For sufficiently large $n$, the connected graph $G$ on $n$
  vertices with largest principal ratio is a kite graph.
\end{theorem}
\addtocounter{theorem}{-1}
}
\begin{proof}
  It remains to show that $C$ induces a clique. Assume it does not, and let $H$ be the graph $P_k \cdot K_{n-k+1}$. We will show
  that $\gamma(H) > \gamma(G)$, and this contradiction tells us that $C$ is a clique. As before, lemma \ref{path_bound} gives that
\[
\gamma(H)  = \frac{\sigma_H^k - \sigma_H^{-k}}{\sigma_H - \sigma_H^{-1}},
\]
where
\[
\sigma(H) = \frac{\lambda_1(H) - \sqrt{\lambda_1(H)^2 -4}}{2}.
\]

Since $x_1,\cdots x_k$ form a pendant path we also know that 
\[
\gamma(G) = \frac{\sigma^k - \sigma^{-k}}{\sigma - \sigma^{-1}}.
\]

Now, $\lambda_1(H) > \lambda_1(G)$ because $E(G) \subsetneq E(H)$. Since the functions $g(x) = x+\sqrt{x^2-4}$ and $f(x) = (x^k - x^{-k})/(x-x^{-1})$ are increasing when $x\geq 1$, we have $\gamma(H) > \gamma(G)$.

 \end{proof}

\section*{Acknowledgements}

We would like to thank greatly Xing Peng for helpful discussions and comments on an earlier draft of this paper.

\bibliographystyle{plain}
\bibliography{bib}

\end{document}